\documentclass[11pt,a4paper]{article}

\usepackage[T1]{fontenc}
\usepackage[utf8]{inputenc}
\usepackage{lmodern}
\usepackage{amsmath,amssymb,amsthm,amsfonts}
\usepackage{mathtools}
\usepackage{geometry}
\usepackage[hidelinks]{hyperref}
\usepackage{enumitem}
\newtheorem{theorem}{Theorem}[section]
\newtheorem{lemma}[theorem]{Lemma}
\newtheorem{proposition}[theorem]{Proposition}
\newtheorem{corollary}[theorem]{Corollary}

\theoremstyle{definition}

\newcommand{\Q}{\mathbb{Q}}

\newcommand{\C}{\mathbb{C}}

\newcommand{\CP}{\mathbb{CP}}
\newcommand{\HP}{\mathbb{HP}}
\newcommand{\Ahat}{\widehat{A}}
\newcommand{\Spin}{\mathrm{Spin}}
\newcommand{\Ric}{\mathrm{Ric}}
\newcommand{\Scal}{\mathrm{Scal}}
\newcommand{\eps}{\varepsilon}

\title{The Kernel of the $\Ahat$-Genus in Rational Spin Bordism is Generated by Ricci-Positive Manifolds}
\author{Gerald H\"ohn\\
\small Kansas State University, USA\\
\small \texttt{gerald@monstrous-moonshine.de}
\and
Philipp H\"ohn\\
\small Universit\"at Bonn, Germany\\
\small \texttt{philipp.hoehn@uni-bonn.de}}
\date{\today}

\begin{document}
\maketitle

\begin{abstract}
We prove that, in every degree, the rational Spin bordism classes represented by manifolds admitting metrics with positive Ricci curvature span exactly the kernel of the $\Ahat$-genus.
More precisely, for
\[
R=\Omega_*^{\Spin}\otimes\Q,\qquad
J=\ker(\Ahat:R\longrightarrow\Q[u]),
\]
the $\Q$-span of bordism classes of Ricci-positive Spin manifolds equals $J$ in each degree.
This answers, 
in the differentiable rational Spin category, 
a question about rational bordism obstructions to positive Ricci curvature which was raised in the context of complex elliptic genera.

The proof uses smooth complete intersections of an odd number $\ell$ of quadrics
\[
Y_{m,\ell}\subset \CP^{2m+\ell},
\qquad
\ell=1,\, 3,\, \ldots,\, 2m-1.
\]
These manifolds have real dimension $4m$, are Spin and Fano, and therefore admit metrics with positive Ricci curvature. A first-order thickening of the $\Ahat$-genus induces $m-1$ linear functionals on $(J/J^2)_{4m}$. Their values on the classes $[Y_{m,\ell}]$ are governed by polynomials $P_{m,q}(\ell)$ of strictly increasing degrees $q+1=1$, $2$, $\ldots$, $m-1$. 
This gives full rank by a polynomial-interpolation argument.
\end{abstract}


\section{Introduction}

In a footnote on page~57 of the first author's 1991 diploma thesis, written in the context of complex elliptic genera of level $N$, the following question was raised \cite[p.~57, fn.~4]{HoehnDiplom}:
\begin{quote}\small
It is known that the $\widetilde A_N$-genus of a complex \(N\)-manifold vanishes if it admits a
K\"ahler metric with positive Ricci curvature; cf. \cite[Theorem~11.15]{LawsonMichelsohn}. Are there,
besides \(\widetilde A_N\), further rational cobordism-theoretic obstructions? In the case of
oriented Spin manifolds, the \(\widehat A\)-genus is presumably the only obstruction of this kind
to metrics with positive Ricci curvature. Indeed, by Yau's proof of the Calabi conjecture,
complete intersections with \(c_1>0\) admit metrics with positive Ricci curvature; cf.~\cite{Yau}. These complete intersections probably generate the kernel of \(\widehat A\) in
\(\Omega_*^{\mathrm{Spin}}\otimes \mathbb Q\).
\end{quote}
Here $\widetilde A_N$ denotes the cusp genus occurring in the theory of complex elliptic genera of level~$N$. 
It is the genus corresponding to the combined twisted Todd genera ${\rm \chi}(X,K^{k/N})$, $k=1$, $\ldots$, $N-1$.
Specializing this question to $N=2$ leads to the Spin situation considered below. 
In the complex-bordism formulation, the odd complex-dimensional part is immediate:
the spaces $\CP^{2k+1}$ satisfy the level-$2$ condition $2|c_1$,
carry Fubini--Study metrics with positive Ricci curvature, and have nonzero Milnor numbers.
The nontrivial part is therefore the even complex-dimensional, equivalently real $4m$-dimensional, Spin part.
The present paper does not use complex structures except to supply explicit Ricci-positive representatives.
The theorem itself is formulated purely in differentiable Spin bordism and proves this remaining assertion in the differentiable rational Spin bordism category.

\smallskip
This is also related to the loop-space interpretation of elliptic genera and what is often called the H\"ohn--Stolz conjecture:
a string manifold with positive Ricci curvature should have vanishing Witten genus; see \cite{StolzWitten}.
The result proved here is an easier Spin-bordism analogue in a converse direction.
\medskip

From the broader viewpoint of curvature versus topology, positive
Ricci curvature occupies an intermediate position between positive
sectional and positive scalar curvature. Positive scalar curvature has
a well-developed surgery and index-theoretic existence theory; in the
simply connected Spin case, Stolz's theorem~\cite{Stolz} identifies the relevant
index obstruction. For positive Ricci curvature, one still has the
Dirac obstruction coming from positive scalar curvature, and global
restrictions such as the Bonnet--Myers finiteness of the fundamental
group, but no comparable surgery classification is known. Thus it is
natural to ask what remains after passing to a coarse invariant such
as rational Spin bordism. The theorem below answers this question:
at the level of rational Spin bordism, the Lichnerowicz obstruction is
the only obstruction. Equivalently, every rational Spin genus which
vanishes on all positive-dimensional Ricci-positive Spin manifolds factors through the
\(\widehat A\)-genus.

\smallskip
The proof is in the spirit of Hirzebruch's characteristic-power-series
method: a square-zero deformation of the \(\widehat A\)-series produces
the coefficient functionals needed to detect the indecomposable part of
\(\ker \widehat A\).

\bigskip
All manifolds considered are smooth and closed. A genus will mean a homogeneous unital $\Q$-algebra homomorphism from a rational bordism ring to a graded $\Q$-algebra~\cite{Hirzebruch}. The rational Spin bordism ring is a polynomial algebra
\[
R:=\Omega_*^{\Spin}\otimes\Q
\cong
\Q[x_4,x_8,x_{12},\ldots],
\qquad \deg(x_{4k})=4k,
\]
with one algebraically independent generator in each positive degree divisible by $4$ \cite{ABP}. The $\Ahat$-genus is the genus
\[
\Ahat:R\longrightarrow \Q[u],
\qquad \deg(u)=4,
\]
with characteristic power series
\[
Q_{\Ahat}(\sqrt{u}\,x)=\frac{\sqrt{u}\,x/2}{\sinh(\sqrt{u}\,x/2)}.
\]
Set
\[
J:=\ker(\Ahat)\subset R.
\]

For $n>0$, let
\[
I_n^{\Spin,\Ric>0}\subset \Omega_n^{\Spin}\otimes\Q
\]
be the $\Q$-subspace spanned by bordism classes of Spin $n$-manifolds admitting metrics with positive Ricci curvature. We put $I_0^{\Spin,\Ric>0}=0$ and set
\[
I^{\Spin,\Ric>0}:=\bigoplus_{n\ge 0}I_n^{\Spin,\Ric>0}\subset R.
\]
Products of Ricci-positive manifolds are Ricci-positive with the product metric, so $I^{\Spin,\Ric>0}$ is a graded nonunital $\Q$-subalgebra of $R$. The Lichnerowicz theorem gives
\[
I^{\Spin,\Ric>0}\subseteq J,
\]
since positive Ricci curvature implies positive scalar curvature.

The main theorem is a kind of converse in rational Spin bordism.

\begin{theorem}\label{thm:main}
For every $n\ge 0$ one has
\[
I_n^{\Spin,\Ric>0}=J_n.
\]
Consequently, $I^{\Spin,\Ric>0}=J$ as graded $\Q$-vector spaces and nonunital graded $\Q$-algebras. In particular, $I^{\Spin,\Ric>0}$ is an ideal of $R$.
\end{theorem}
The assertion is about the rational span of such bordism classes; it does not assert that every integral class in $J$ has a connected Ricci-positive representative.

The proof is degreewise. In degree $4m$, the quotient $(J/J^2)_{4m}$ has dimension $m-1$. We construct $m$ Ricci-positive Spin manifolds of dimension $4m$ and show that their images generate this quotient. The manifolds are complete intersections of an odd number of quadrics. The linear independence is detected by a square-zero deformation of the $\Ahat$-genus. Dessai previously verified the corresponding spanning statement by computer calculation in a finite range of dimensions~\cite{Dessai} (see also~\cite{Reiser2024}); the argument below gives a uniform proof in all degrees.

\medskip

We benefited at several points from discussions with GPT-5.5 Pro. All mathematical arguments and computations were verified by the authors.


\section{Geometric background}

This section records the geometric facts used in the proof. The only analytic input needed for the construction is Yau's prescribed Ricci form theorem and the Lichnerowicz vanishing theorem.


\subsection{Products and positive Ricci curvature}

\begin{proposition}\label{prop:product-ricci}
If $M$ and $N$ admit Riemannian metrics with $\Ric>0$, then $M\times N$ admits a Riemannian metric with $\Ric>0$.
\end{proposition}

\begin{proof}
Let $g_M$ and $g_N$ be metrics with positive Ricci curvature, and equip $M\times N$ with the product metric
\[
g=g_M\times g_N.
\]
The Levi-Civita connection and curvature tensor split. For $v=(v_M,v_N)\in T_{(p,q)}(M\times N)$ one has
\[
\Ric_g(v,v)=\Ric_{g_M}(v_M,v_M)+\Ric_{g_N}(v_N,v_N).
\]
If $v\ne 0$, then at least one component is nonzero, and the corresponding summand is positive. Hence $\Ric_g>0$.
\end{proof}

In contrast, a product has positive scalar curvature if one factor has positive scalar curvature and is rescaled sufficiently.
This stronger stability is not available for positive Ricci curvature, which is why the proof below constructs multiple representatives degree by degree.


\subsection{The Lichnerowicz theorem}

\begin{proposition}[Lichnerowicz~\cite{Lichnerowicz}]\label{prop:lichnerowicz}
If $M$ is a Spin manifold admitting a metric with positive scalar curvature, then
\[
\Ahat(M)=0.
\]
In particular, every Ricci-positive Spin manifold has vanishing $\Ahat$-genus.
\end{proposition}

\begin{proof}
Let $D$ be the Spin Dirac operator. The Lichnerowicz formula~\cite{Lichnerowicz} is
\[
D^2=\nabla^*\nabla+\frac14\Scal.
\]
For a spinor $\psi$, integration gives
\[
\|D\psi\|_{L^2}^2
=
\|\nabla\psi\|_{L^2}^2
+\frac14\int_M \Scal\,|\psi|^2\,d\mu.
\]
If $\Scal>0$ and $D\psi=0$, then both terms on the right vanish and therefore $\psi=0$. Thus the Dirac operator has zero index. By the Atiyah-Singer index theorem, in real dimension $4m$ this index is the coefficient of $u^m$ in $\Ahat(M)$ \cite{AtiyahSinger}. Since $\Ric>0$ implies $\Scal>0$, the final assertion follows.
\end{proof}

Stolz's theorem gives the corresponding existence theorem for scalar curvature: if $M$ is simply connected, Spin, and of dimension at least $5$, then $M$ admits a metric of positive scalar curvature if and only if the $\alpha$-invariant in $KO_{\dim M}(\mathrm{pt})$ vanishes~\cite{Stolz}. In dimensions divisible by $4$, the rational part of this obstruction is detected by the $\Ahat$-genus. The present theorem is a rational Spin-bordism analogue for positive Ricci curvature. It is unknown whether surgery or other methods can provide results for positive Ricci curvature comparable to those known for positive scalar curvature.


\subsection{Fano manifolds and positive Ricci curvature}

\begin{proposition}[Yau's prescribed Ricci form theorem]\label{prop:yau-prescribed}
Let $X$ be a compact K\"ahler manifold. For every real closed $(1,1)$-form $\rho$ representing the real class $2\pi c_1(X)$ and every K\"ahler class $[\omega_0]$, there exists a unique K\"ahler metric $\omega\in[\omega_0]$ such that
\[
\Ric(\omega)=\rho.
\]
\end{proposition}

\begin{proof}[Proof (Sketch)]
The forms $\Ric(\omega_0)$ and $\rho$ represent the same cohomology class, so
\[
\Ric(\omega_0)-\rho=\sqrt{-1}\,\partial\bar\partial h
\]
for a smooth real function $h$. Yau's solution of the Calabi conjecture solves the complex Monge--Amp\`ere equation
\[
(\omega_0+\sqrt{-1}\,\partial\bar\partial\varphi)^n=C e^h\omega_0^n,
\qquad
\omega_0+\sqrt{-1}\,\partial\bar\partial\varphi>0,
\]
where $n=\dim_{\C}X$ and $C>0$ is chosen so that the total volumes agree \cite{Yau}. The metric $\omega=\omega_0+\sqrt{-1}\,\partial\bar\partial\varphi$ then satisfies $\Ric(\omega)=\rho$.
\end{proof}

\begin{corollary}\label{cor:fano-ricci}
Every smooth Fano manifold admits a K\"ahler metric with $\Ric>0$.
\end{corollary}

\begin{proof}
Recall that a compact complex manifold \(X\) is called Fano if the determinant line bundle
\(\det_{\C}(TX)=K_X^*\) is ample.
Equivalently, $c_1(X)$ is represented by a positive real $(1,1)$-form. Choose a positive representative $\rho\in 2\pi c_1(X)$ and apply Proposition \ref{prop:yau-prescribed}. Positivity of the Ricci form is positivity of the Ricci tensor for the underlying Riemannian metric.
\end{proof}


\subsection{Quadric complete intersections}

Let \(\gamma_N\longrightarrow\CP^N\) be the tautological complex line bundle,
whose fibre over a point \([L]\in\CP^N\) is the line \(L\subset\C^{N+1}\).
We normalize
\[
g=c_1({\gamma_N^*})\in H^2(\CP^N;\mathbb Z)
\]
so that \(g\) restricts to the positive generator of \(H^2(\CP^1;\mathbb Z)\)
for the complex orientation of \(\CP^1\).

For integers \(m\geq 2\) and \(1\leq \ell\leq 2m-1\), let
\[
Y_{m,\ell}=X_{\underbrace{\scriptstyle 2,\ldots,2}_{\ell}}\subset \CP^{2m+\ell}
\]
be a smooth complete intersection of \(\ell\) quadrics. Such smooth complete
intersections exist: choosing the equations successively among quadratic
homogeneous equations, equivalently among sections of
\((\gamma_{2m+\ell}^*)^{\otimes 2}\), and applying Bertini's theorem at each
step gives a smooth complete intersection; equivalently, one may apply
Bertini after the quadratic Veronese embedding; see
Hartshorne~\cite[II, Theorem 8.18]{Hartshorne}.

\begin{proposition}
If \(\ell\) is odd and \(1\leq \ell\leq 2m-1\), then \(Y_{m,\ell}\)
is Spin and Fano. Consequently \(Y_{m,\ell}\) admits a metric with
\(\operatorname{Ric}>0\), and
\[ [Y_{m,\ell}]\in J_{4m}. \]
\end{proposition}

\begin{proof}
Put \(N=2m+\ell\), and let $i:Y_{m,\ell}\hookrightarrow \CP^N$ be the inclusion.
Since \(Y_{m,\ell}\) is cut out transversely by \(\ell\) quadratic equations,
its complex normal bundle is
\[
\nu
\cong
\bigoplus_{j=1}^{\ell}{\gamma_N^*}^{\otimes 2}|_{Y_{m,\ell}}.
\]
The tangent-normal exact sequence is
\[
0\longrightarrow TY_{m,\ell}
\longrightarrow T\CP^N|_{Y_{m,\ell}}
\longrightarrow \nu
\longrightarrow 0.
\]
Taking complex determinants of this exact sequence gives
\[
K_{Y_{m,\ell}}^*
=
\det\nolimits_{\C}(TY_{m,\ell})
\cong
{\gamma_N^*}^{\otimes(N+1-2\ell)}|_{Y_{m,\ell}}
=
{\gamma_N^*}^{\otimes(2m+1-\ell)}|_{Y_{m,\ell}}
\]
Hence
\[
c_1(TY_{m,\ell})
=
(2m+1-\ell)\,c_1(\gamma_N^*|_{Y_{m,\ell}})
=
(2m+1-\ell)\,i^*g.
\]

\smallskip

Since \(w_2\equiv c_1\pmod 2\) for complex manifolds, \(Y_{m,\ell}\)
is Spin when \(\ell\) is odd. If \(\ell\leq 2m-1\), then
\(2m+1-\ell>0\), so \(K_{Y_{m,\ell}}^*\) is a positive tensor power of
\({\gamma_N^*}|_{Y_{m,\ell}}\). Thus \(K_{Y_{m,\ell}}^*\) is ample, and \(Y_{m,\ell}\)
is Fano.

\smallskip

Corollary~\ref{cor:fano-ricci} gives a metric with \(\operatorname{Ric}>0\), and
Proposition~\ref{prop:lichnerowicz} gives \(\widehat A(Y_{m,\ell})=0\). Therefore
\([Y_{m,\ell}]\in J_{4m}\).
\end{proof}

\section{Proof of Theorem}

\subsection{The quotient \texorpdfstring{$J/J^2$}{J/J2}}

We first explain the elementary algebra needed for the induction. Since a K3
surface, for example a quartic in $\CP^3$, has the $\Ahat$-genus $\Ahat(K3)=2u$, the genus
\[
\Ahat:R\longrightarrow \Q[u]
\]
is surjective, and hence $R/J\cong\Q[u]$.

\begin{proposition}\label{prop:J-J2}
The $R/J\cong\Q[u]$-module $J/J^2$ is free with one generator in each degree
$4k$, $k\ge 2$. In particular,
\[
J_0=J_4=0,\qquad
\dim_\Q (J/J^2)_{4m}=m-1
\]
for every $m\ge 2$.
\end{proposition}

\begin{proof}
Choose homogeneous polynomial generators
\[
R=\Q[z_4,z_8,z_{12},\ldots],\qquad \deg(z_{4k})=4k,
\]
with $\Ahat(z_4)=u$. For $k\ge 2$ write
\[
\Ahat(z_{4k})=a_k u^k,\qquad a_k\in\Q,
\]
and set
\[
y_{4k}:=z_{4k}-a_k z_4^k.
\]
This triangular change gives
\[
R=\Q[z_4,y_8,y_{12},\ldots],
\]
and all $y_{4k}$ lie in $J$. Since $z_4$ maps to $u$, it follows that
\[
J=\langle y_8,y_{12},y_{16},\ldots \rangle
\]
and therefore
\[
J/J^2\cong \bigoplus_{k\ge 2}\Q[u]\cdot \overline y_{4k}.
\]
The degree $4m$ part has basis
\[
u^{m-k}\,\overline y_{4k},\qquad 2\le k\le m,
\]
so its dimension is $m-1$. The statements $J_0=J_4=0$ are immediate from the
same description.
\end{proof}
More concretely, one may take the $y_{4k}$ to be the bordism classes of the 
quaternionic projective spaces $\HP^{k}$, $k\geq 2$.

\begin{proposition}\label{prop:algebraic-reduction}
Assume that for every $m\ge 2$ there exist Ricci-positive Spin manifolds
$Z_{m,1},\ldots,Z_{m,N_m}$ of real dimension $4m$ whose classes span
$(J/J^2)_{4m}$. Then Theorem~\ref{thm:main} holds.
\end{proposition}

\begin{proof}
Put $I:=I^{\Spin,\Ric>0}$. Proposition~\ref{prop:lichnerowicz} gives
$I\subseteq J$. We prove $J\subseteq I$ by induction on the degree. The cases
of degrees $0$ and $4$ follow from Proposition~\ref{prop:J-J2}.

Assume $m\ge 2$ and $I_e=J_e$ for all $e<4m$. Then
\[
(J^2)_{4m}\subseteq I_{4m},
\]
because every product contributing to $(J^2)_{4m}$ has both factors in positive
degrees smaller than $4m$, hence in $I$ by induction; and $I$ is a graded
subalgebra.

Let $x\in J_{4m}$. By the spanning hypothesis there are $a_1$, $\ldots$, $a_{N_m}\in\Q$ such that
\[
x-\sum_{i=1}^{N_m}a_i[Z_{m,i}]\in (J^2)_{4m}.
\]
The classes $[Z_{m,i}]$ lie in $I_{4m}$, and the remaining term lies in
$I_{4m}$ by the previous paragraph. Hence $x\in I_{4m}$. This completes the
induction.
\end{proof}

\subsection{The thickened genus}

Let
\[
S:=\Q[u]\oplus \eps\,\Q[u,v]\subset \Q[u,v,\eps]/\langle \eps^2\rangle,
\qquad
\deg(u)=\deg(v)=\deg(\eps)=4.
\]
Define a genus $\Phi:R\longrightarrow S$ by the characteristic power series
\[
Q_\Phi(x)=Q_{\Ahat}(\sqrt u\,x)\left(1+\eps\rho(x)\right),
\qquad
\rho(x):=\frac{vx^4}{1-vx^2}.
\]
Write
\[
\Phi(X)=\Ahat(X)+\eps E(X).
\]

\begin{lemma}\label{lem:E-Qulinear}
If $X\in J$, then for all $Y\in R$,
\[
E(XY)=\Ahat(Y)E(X).
\]
Consequently $E$ vanishes on $J^2$ and induces a $\Q[u]$-linear map
\[
\overline E:J/J^2\longrightarrow \Q[u,v].
\]
\end{lemma}

\begin{proof}
Multiplicativity and $\eps^2=0$ give
\[
\Phi(XY)=(\Ahat(X)+\eps E(X))(\Ahat(Y)+\eps E(Y)),
\]
so
\[
E(XY)=\Ahat(X)E(Y)+\Ahat(Y)E(X).
\]
If $X\in J$, then $\Ahat(X)=0$, giving the displayed formula. If also
$Y\in J$, then $E(XY)=0$, so $E(J^2)=0$. The same formula gives
$\Q[u]$-linearity after identifying $R/J$ with $\Q[u]$.
\end{proof}

For $x\in (J/J^2)_{4m}$, the polynomial $\overline E(x)$ is homogeneous of
degree $4m-4$. Since the deformation term $\rho$ is divisible by $v$, there is a unique expansion
\[
\overline E(x)=
\sum_{r=0}^{m-2}\Lambda_{m,r}(x)\,u^{m-r-2}v^{r+1},
\qquad
\Lambda_{m,r}(x)\in\Q.
\]
Thus the thickened genus gives $m-1$ coefficient functionals
\[
\Lambda_{m,r}:(J/J^2)_{4m}\longrightarrow \Q,
\qquad 0\le r\le m-2.
\]
Since $\dim_\Q(J/J^2)_{4m}=m-1$, a family of classes in $(J/J^2)_{4m}$
spans as soon as its matrix of $\Lambda_{m,r}$-values has rank $m-1$.


\subsection{Evaluation on complete intersections of quadrics}

Recall that $\gamma_N$ denotes the tautological line bundle on
$\CP^N$ and $g=c_1(\gamma_N^*)$. We write $[z^N]F(z)$ for the coefficient of $z^N$ of a power series $F(z)$.
If
\[
X=X_{d_1,\ldots,d_c}\subset\CP^N
\]
is a smooth complete intersection of hypersurfaces of degrees $d_1$, $\ldots$, $d_c$,
then the genus with characteristic power series $Q(x)$ equals
\[
\Phi(X)
=
[g^N]\left(
\left(\prod_{i=1}^c d_i \right)g^c
\frac{Q(g)^{N+1}}{\prod_{i=1}^c Q(d_i g)} \right).
\]
This follows from the sequence
\[
0\longrightarrow TX
\longrightarrow T\CP^N|_X
\longrightarrow \bigoplus_{i=1}^c {\gamma_N^*}^{\otimes d_i}|_X
\longrightarrow 0,
\]
for the normal bundle of $X$ and from the fact that the fundamental class of $X$ is Poincar\'e dual to
$\left(\prod_{i=1}^c d_i\right)g^c$.

For
\[
Y_{m,\ell}=X_{\underbrace{\scriptstyle 2,\ldots,2}_{\ell}}\subset\CP^{2m+\ell}
\]
we get
\[
\Phi(Y_{m,\ell})
=
[g^{2m+\ell}]
\left(
2^\ell g^\ell\,
\frac{Q_\Phi(g)^{2m+\ell+1}}{Q_\Phi(2g)^\ell}
\right).
\]

For $0\le q\le m-2$, define polynomials in the variable $n$ by
\[
P_{m,q}(n)
=
\left(2m+1+n(1-4^{m-q})\right)
[s^{2q}]
Q_{\Ahat}(2s)^{2m+1}\cosh(s)^n .
\]

\smallskip
\begin{proposition}\label{prop:Lambda-eval}
Let $m\ge 2$, let $0\le r\le m-2$, let $\ell$ be odd with
$1\le \ell\le 2m-1$, and set $q=m-r-2$. Then
\[
\Lambda_{m,r}([Y_{m,\ell}])
=
2^\ell 4^{-q}P_{m,q}(\ell).
\]
\end{proposition}

\begin{proof}
Since $\eps^2=0$,
\[
\frac{Q_\Phi(g)^{2m+\ell+1}}{Q_\Phi(2g)^\ell}
=
\frac{Q_{\Ahat}(\sqrt u\,g)^{2m+\ell+1}}
     {Q_{\Ahat}(2\sqrt u\,g)^\ell}
\left(
1+\eps\,\bigl((2m+\ell+1)\rho(g)-\ell\,\rho(2g)\bigr)
\right).
\]
Since
\[
\rho(g)=\sum_{a\ge 0}v^{a+1}g^{2a+4},
\qquad
\rho(2g)=\sum_{a\ge 0}v^{a+1}4^{a+2}g^{2a+4},
\]
the coefficient of $v^{r+1}$ in the deformation term is
\[
\left(2m+1+\ell(1-4^{r+2})\right)g^{2r+4}.
\]
With $s=\sqrt u\,g/2$ one has by the double-angle identity for the hyperbolic sine
\[
\frac{Q_{\Ahat}(\sqrt u\,g)^{2m+\ell+1}}
     {Q_{\Ahat}(2\sqrt u\,g)^\ell}
=
Q_{\Ahat}(2s)^{2m+1}\cosh(s)^\ell.
\]
Thus the part of this factor contributing $g^{2q}$ is
\[
4^{-q}u^q
[s^{2q}]Q_{\Ahat}(2s)^{2m+1}\cosh(s)^\ell .
\]
The total power of $g$ is correct exactly when
\[
\ell+(2r+4)+2q=2m+\ell,
\]
i.e. $q=m-r-2$. Therefore the coefficient of
$\eps u^{m-r-2}v^{r+1}$ is
\[
2^\ell4^{-q}
\left(2m+1+\ell(1-4^{r+2})\right)
[s^{2q}]Q_{\Ahat}(2s)^{2m+1}\cosh(s)^\ell .
\]
Since $r+2=m-q$, this is $2^\ell4^{-q}P_{m,q}(\ell)$.
\end{proof}

\subsection{The rank argument and completion of the proof}

\begin{proposition}\label{prop:evaluation-rank}
For every $m\ge 2$, the matrix
\[
\bigl(\Lambda_{m,r}([Y_{m,\ell}])\bigr)_{
\ell\in\{1,3,\ldots,2m-1\},\;0\le r\le m-2}
\]
has rank $m-1$.
\end{proposition}

\begin{proof}
By Proposition~\ref{prop:Lambda-eval}, after multiplying rows and columns
by nonzero scalars and reordering the columns via $q=m-r-2$, it is enough
to prove the rank statement for
\[
\bigl(P_{m,q}(\ell)\bigr)_{
\ell\in\{1,3,\ldots,2m-1\},\;0\le q\le m-2}.
\]
We claim that, as a polynomial in $n$,
\[
\deg P_{m,q}=q+1.
\]
Indeed, write
\[
\cosh(s)^n=\exp(nB(s)),\qquad
B(s)=\log\cosh(s)=\frac{s^2}{2}+O(s^4).
\]
Then
\[
[s^{2q}]Q_{\Ahat}(2s)^{2m+1}\cosh(s)^n
\]
is a polynomial in \(n\) of degree at most \(q\), and
\[
[n^q][s^{2q}]Q_{\Ahat}(2s)^{2m+1}\cosh(s)^n
=
[s^{2q}]\frac{B(s)^q}{q!}
=
\frac{1}{2^q q!}.
\]
Here, only the constant term \(1\) of \(Q_{\Ahat}(2s)^{2m+1}\) contributes to the \(n^q\)-coefficient.
Therefore
\[
[n^{q+1}]P_{m,q}(n)
=
\frac{1-4^{m-q}}{2^q q!}\ne 0,
\]
since \(m-q\ge 2\).

Now suppose
\[
\sum_{q=0}^{m-2} b_qP_{m,q}(\ell)=0
\qquad
\text{for } \ell=1,\, 3,\, \ldots,\, 2m-1.
\]
The left-hand side is a polynomial in $\ell$ of degree at most $m-1$
with $m$ distinct zeros, hence it is identically zero. Since
$P_{m,0}$, $\ldots$, $P_{m,m-2}$ have strictly increasing degrees
$1$, $\ldots$, $m-1$, all $b_q$ vanish. The columns are therefore linearly independent.
\end{proof}

It follows that the images of
\[
[Y_{m,1}],\, [Y_{m,3}],\, \ldots,\, [Y_{m,2m-1}]
\]
span $(J/J^2)_{4m}$. Since these classes are represented by
Ricci-positive Spin manifolds, Proposition~\ref{prop:algebraic-reduction}
applies in every degree $4m$. Together with $J_0=J_4=0$ and the rational
vanishing of $\Omega_n^{\Spin}\otimes\Q$ in degrees~$n$ not divisible by $4$,
this proves Theorem~\ref{thm:main}.

\end{document}